\documentclass[10pt,english,leqno]{article}
\usepackage[T1]{fontenc}
\usepackage{times}
\usepackage{url}
\usepackage[latin9]{inputenc}
\usepackage{amsthm}
\usepackage{amsmath}
\usepackage{amssymb}
\usepackage{enumerate}
\usepackage{todonotes}
\usepackage[numbers]{natbib}
\usepackage{hyperref}

\makeatletter
\newtheorem{thm}{Theorem}
\newtheorem{lemma}[thm]{Lemma}

\theoremstyle{definition}

\newtheorem{conj}[thm]{Conjecture}

\newsavebox{\fmbox}

\newlength{\dyindent}
{\end{list}}

\newenvironment{dy*}{\refstepcounter{equation}\begin{list}{}%
{\setlength{\leftmargin}{\dyindent}\setlength{\labelwidth}{\dyindent}%
\addtolength{\labelwidth}{-\labelsep}}%
\item}%
{\end{list}}

\newcommand{\nullity}{\mbox{nullity}}

\title{A graph minors characterization of signed graphs whose signed Colin de Verdi\`ere parameter $\nu$ is two}

\author{Marina Arav, Frank J. Hall, Zhongshan Li, Hein van der Holst\footnote{Corresponding author, E-mail: hvanderholst@gsu.edu} \\
Department of Mathematics and Statistics \\
Georgia State University \\
Atlanta, GA 30303, USA
}

\date{}
\begin{document}

\maketitle

\begin{abstract}
A signed graph is a pair $(G,\Sigma)$, where $G=(V,E)$ is a graph (in which parallel edges are permitted, but loops are not) with $V=\{1,\ldots,n\}$ and $\Sigma\subseteq E$. The edges in $\Sigma$ are called odd and the other edges even. By $S(G,\Sigma)$ we denote the set of all symmetric $n\times n$ matrices $A=[a_{i,j}]$ with $a_{i,j}<0$ if $i$ and $j$ are connected by only even edges, $a_{i,j}>0$ if $i$ and $j$ are connected by only odd edges, $a_{i,j}\in \mathbb{R}$ if $i$ and $j$ are connected by both even and odd edges, $a_{i,j}=0$ if $i\not=j$ and $i$ and $j$ are non-adjacent, and $a_{i,i} \in \mathbb{R}$ for all vertices $i$. The parameter $\nu(G,\Sigma)$ of a signed graph $(G,\Sigma)$ is the largest nullity of any positive semidefinite matrix $A\in S(G,\Sigma)$ that has the Strong Arnold Property. By $K_3^=$ we denote the signed graph obtained from $(K_3,\emptyset)$ by adding to each even edge an odd edge in parallel. In this paper, we prove that a signed graph $(G,\Sigma)$ has $\nu(G,\Sigma)\leq 2$ if and only if $(G,\Sigma)$ has no minor isomorphic to $(K_4,E(K_4))$ or $K_3^=$. 
\end{abstract}

\section{Introduction}
A \emph{signed graph} is a pair $(G,\Sigma)$, where $G=(V,E)$ is a graph (in which parallel edges are permitted, but loops are not)  with $V=\{1,\ldots,n\}$ and $\Sigma\subseteq E$. (We refer to \cite{Diestel} for the notions and concepts in Graph Theory.) The edges in $\Sigma$ are called \emph{odd} and the other edges \emph{even}. If $V=\{1,2,\ldots,n\}$, we denote by $S(G,\Sigma)$ the set of all real symmetric $n\times n$ matrices $A=[a_{i,j}]$ with 
\begin{itemize}
\item $a_{i,j} < 0$ if $i$ and $j$ are connected by only even edges, \item $a_{i,j}>0$ if $i$ and $j$ are connected by only odd edges, 
\item $a_{i,j}\in \mathbb{R}$ if $i$ and $j$ are connected by both even and odd edges, 
\item $a_{i,j}=0$ if $i\not=j$ and $i$ and $j$ are non-adjacent, and 
\item $a_{i,i} \in \mathbb{R}$ for all vertices $i$. 
\end{itemize}
In \cite{AraHalLivdH2012} we introduced for any signed graph $(G,\Sigma)$ the signed graph parameter $\nu$. In order to describe this parameter we need the notion of \emph{Strong Arnold Property} (\emph{SAP} for short). A matrix $A=[a_{i,j}]\in S(G,\Sigma)$ has the SAP if $X=0$ is the only symmetric matrix $X=[x_{i,j}]$ such that $x_{i,j} = 0$ if $i$ and $j$ are adjacent vertices or $i=j$, and $A X = 0$. Then $\nu(G,\Sigma)$ is defined as the largest nullity of any positive semidefinite matrix $A\in S(G,\Sigma)$ that has the SAP. 

If $G$ is a graph and $v$ a vertex of $G$, then $\delta(v)$ denotes the set of edges of $G$ incident to $v$. If $G$ is a graph and $U\subseteq V(G)$, then $\delta(U)$ denotes the set of edges of $G$ that have one end in $U$ and one end in $V(G)\setminus U$. 
The symmetric difference of two sets $A$ and $B$ is the set $A\Delta B = A\setminus B\cup B\setminus A$.
If $(G,\Sigma)$ is a signed graph and $U\subseteq V(G)$, we say that $(G,\Sigma)$ and $(G,\Sigma\Delta\delta(U))$ are \emph{sign-equivalent} and call the operation $\Sigma\to \Sigma\Delta\delta(U)$ \emph{re-signing on $U$}. Re-signing on $U$ amounts to performing a diagonal similarity on the matrices in $S(G,\Sigma)$, and hence it does not affect $\nu(G,\Sigma)$. We call a cycle $C$ of a signed graph $(G,\Sigma)$ \emph{odd} if $\Sigma\cap E(C)$ has an odd number of elements, otherwise  we call $C$ \emph{even}. We call a signed graph \emph{bipartite} if it has no odd cycles. Zaslavsky showed in \cite{MR676405} that two signed graphs are sign-equivalent if and only if they have the same set of odd cycles. Thus, two signed graphs $(G,\Sigma)$ and $(G,\Sigma')$ that have the same set of odd cycles have $\nu(G,\Sigma) = \nu(G,\Sigma')$.  

\emph{Contracting} an even edge $e=uv$ in a signed graph $(G,\Sigma)$ means deleting $e$ and identifying the vertices $u$ and $v$, retaining the signs on the other edges. Contracting an odd edge $e=uv$ in a signed graph $(G,\Sigma)$ means first re-signing around $u$ (or $v$) and then contracting $e$ in the resulting signed graph. Note that if $(H,\Sigma')$ is obtained from $(G,\Sigma)$ by contracting an edge $e$, then the sign of each cycle $C$ containing $e$ in $(G,\Sigma)$ is the same as the sign of the cycle $C'$ in $(H,\Sigma')$ obtained from $C$ by contracting $e$. A subgraph of a signed graph is defined similarly as in the graph case. A signed graph $(H,\Sigma')$ is a \emph{minor} of a signed graph $(G,\Sigma)$ if $(H,\Sigma')$ is sign-equivalent to a signed graph that can be obtained from $(G,\Sigma)$ by contracting a sequence of edges in a subgraph of $(G,\Sigma)$. The parameter $\nu$ has a very nice property which is stated in the following theorem.

\begin{thm}\cite{AraHalLivdH2012}\label{thm:numinor}
If $(H,\Omega)$ is a minor of a signed graph $(G,\Sigma)$, then $\nu(H,\Omega)\leq \nu(G,\Sigma)$.
\end{thm} 

In \cite{AraHalLivdH2012} we also proved that a signed graph $(G,\Sigma)$ has $\nu(G,\Sigma)\leq 1$ if and only if $(G,\Sigma)$ is bipartite. For a positive integer $n$, we denote by $K_n^o$ the signed graph $(K_n,E(K_n))$ and by $K_n^=$ the signed graph $(G,E(K_n))$, where $G$ is the graph obtained from $K_n$ by adding to each edge an edge in parallel. 
In Figure~\ref{K4oK3=}, the signed graph $K_4^o$ and $K_3^=$ are depicted; here a bold edge denotes an odd edge and a thin edge an even edge. In this paper we prove the following theorem.

\begin{figure}
\begin{center}
\includegraphics[width=0.7\textwidth]{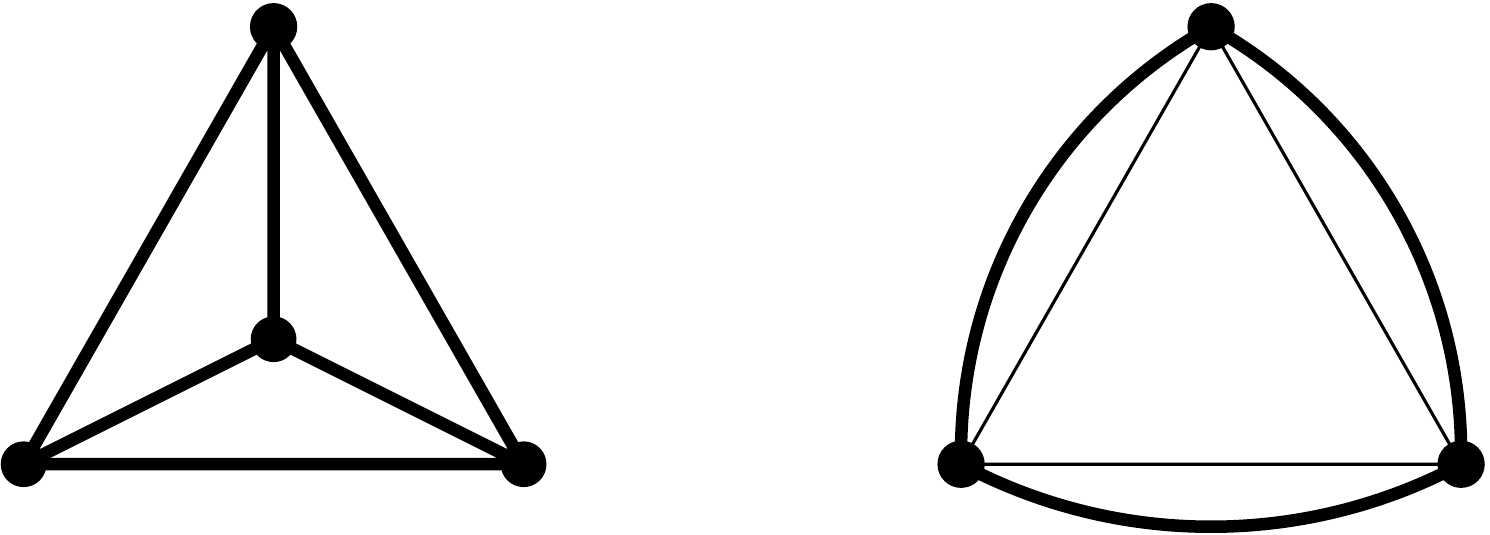}
\end{center}
\caption{$K_4^o$ and $K_3^=$}\label{K4oK3=}
\end{figure}

\begin{thm}\label{thm:mainthm}
A signed graph $(G,\Sigma)$ has $\nu(G,\Sigma)\leq 2$ if and only if $(G,\Sigma)$ has no minor isomorphic to $K_4^o$ or $K_3^=$.
\end{thm} 

One direction of the proof of this theorem follows easily from Theorem~\ref{thm:numinor}: As $\nu(K_4^o) = \nu(K_3^=) = 3$ (see \cite{AraHalLivdH2012}), any signed graph $(G,\Sigma)$ with $\nu(G,\Sigma)\leq 2$ cannot have a minor isomorphic to $K_4^o$ or $K_3^=$. The proof of the opposite direction spans the major part of this paper. In the proof we will use a decomposition theorem of Gerards~\cite{Ger90a, RePEc:dgr:kubrem:1986237}, see the next section for the statement of this theorem. 

The parameter $\nu$ for signed graphs is analogous to the graph parameter $\nu$ introduced by Colin de Verdi\`ere in \cite{CdeV3}.  
For a simple graph $G=(V,E)$ with $V=\{1,2,\ldots,n\}$, denote by $S(G)$ the set of all real symmetric $n\times n$ matrices $A=[a_{i,j}]$ with $a_{i,j}\not=0$ if $i\not=j$ and $i$ and $j$ are adjacent, $a_{i,j}=0$ if $i\not=j$ and $i$ and $j$ are non-adjacent, and $a_{i,i}\in \mathbb{R}$ for all $i\in V$.
For a simple graph $G$, $\nu(G)$ is defined to be the largest nullity of any positive semidefinite matrix $A\in S(G)$ having the SAP. This parameter has the property that if $H$ is a minor of $G$, then $\nu(H)\leq \nu(G)$. 
Colin de Verdi\`ere showed in \cite{CdeV3} that for simple graphs, $\nu(G)\leq 1$ if and only if $G$ is a forest. 
The simple graphs $G$ with $\nu(G)\leq 2$ have been characterized by Kotlov \cite{Kotlov2000a}. The parameter $\nu$ can be extended to graphs in which parallel edges are permitted, but loops are not; see \cite{Holst96a}. In \cite{Holst97a}, van der Holst gave a characterization of graphs $G$ with $\nu(G)\leq 2$. In \cite{Holst2002b}, van der Holst gave a characterization of graphs $G$ with $\nu(G)\leq 3$. 


The \emph{maximum nullity} of $(G,\Sigma)$, denoted $M(G,\Sigma)$, is the maximum of the nullities of the matrices in $S(G,\Sigma)$. 
The \emph{maximum semidefinite nullity} of $(G,\Sigma)$, denoted $M_+(G,\Sigma)$, is the maximum of the nullities of the positive semidefinite matrices in $S(G,\Sigma))$. Clearly, $M_+(G,\Sigma)\leq M(G,\Sigma)$. As $\nu(G,\Sigma)$ is a lower bound of $M_+(G,\Sigma)$, it can be used to obtain lower bounds for $M_+(G,\Sigma)$ using minors of $(G,\Sigma)$. For example, let $K_2^=$ denote the signed graph $(C_2,\{e\})$, where $C_2$ is the $2$-cycle and $e$ an edge of $C_2$. A signed graph $(G,\Sigma)$ has a minor isomorphic to $K_2^=$ if and only if $(G,\Sigma)$ has an odd cycle. In \cite{AraHalLivdH2012}, we proved that $\nu(K_2^=)=2$, and so if a signed graph $(G,\Sigma)$ contains an odd cycle, then $\nu(K_2^=)=2\leq M_+(G,\Sigma)$. Hence a signed graph $G$ with $M_+(G,\Sigma)\leq 1$ does not contain any odd cycle, that is, $(G,\Sigma)$ is bipartite. The complete characterization of signed graphs $(G,\Sigma)$ with $M_+(G,\Sigma)\leq 1$ is as follows: $M_+(G,\Sigma)\leq 1$ if and only if $(G,\Sigma)$ is connected and bipartite. See \cite{AraHalLivdH2012} for a proof.

%
%
%
%
%
%

\section{A decomposition theorem of Gerards}

For presenting the decomposition theorem of Gerards, we need to introduce the notion of $k$-split; see \cite{Ger90a}. 

Let $(G,\Sigma)$ be a signed graph. Suppose $E_1,E_2$ is a partitioning of $E(G)$ with both $E_1$ and $E_2$ nonempty. For $i=1,2$, let $V_i$ be the set of all ends of edges in $E_i$, and let $\tilde{G_i} = (V_i,E_i)$.

Let $k = \lvert V_1\cap V_2\rvert$. If $k\leq 1$, then we say that $(\tilde{G_1}, E_1\cap \Sigma)$ and $(\tilde{G_2}, E_2\cap \Sigma)$ form a \emph{$k$-split} of $(G,\Sigma)$. The signed graphs $(\tilde{G_1}, E_1\cap \Sigma)$ and $(\tilde{G_2}, E_2\cap \Sigma)$ are called the parts of the $k$-split. 

Suppose that $V_1\cap V_2 = \{u,v\}$ with $u\not=v$ and that $\tilde{G_i}$ for $i=1,2$ is connected and not a signed subgraph of $K_2^=$.
Define the signed graph $(G_1,\Sigma_1)$ as follows. If $(\tilde{G_2},E_2\cap \Sigma)$ is not bipartite, add between the vertices $u$ and $v$ of $(\tilde{G_1},E_1\cap \Sigma)$ an odd and an even edge. If $(\tilde{G_2}, E_2\cap \Sigma)$ is bipartite, add a single edge between $u$ and $v$. Add $e$ to $\Sigma_1$ if and only if there exists an odd path between $u$ and $v$ in $(\tilde{G_2},E_2\cap\Sigma)$. The signed graph $(G_2,\Sigma_2)$ is defined similarly. We say that $(G_1,\Sigma_1)$ and $(G_2,\Sigma_2)$ form a \emph{$2$-split} of $(G,\Sigma)$. The signed graphs $(G_1, \Sigma_1)$ and $(G_2, \Sigma_2)$ are called the parts of the $2$-split. If $(\tilde{G_i}, E_i\cap \Sigma)$ is not bipartite for $i=1,2$, then we call the $2$-split \emph{strong}.

Suppose $|V_1\cap V_2|=3$, say $V_1\cap V_2=\{u_1,u_2,u_3\}$. Furthermore, assume that $\tilde{G_2}$ is bipartite and connected, and that $|E_2|\geq 4$. Define $G_1$ as follows. Let $w$ be a new node, let $V(G_1)=V_1\cup \{w\}$ and $E(G_1)=E_1\cup \{u_1w, u_2w, u_3w\}$. Define $\tilde{\Sigma}$ to be the subset of $\{u_2w, u_3w\}$ which has $u_iw\in \tilde{\Sigma}$ if and only if there exists an odd path from $u_1$ to $u_i$ in $(\tilde{G_2}, E_2\cap \Sigma)$. Define $\Sigma_1=(E_1\cap \Sigma)\cup \tilde{\Sigma}$. We say that $(G_1, \Sigma_1)$ forms a \emph{$3$-split} of $(G,\Sigma)$. We call $(G_1,\Sigma_1)$ the part of the $3$-split (so a $3$-split has only one part).

A signed graph $(G,\Sigma)$ is called \emph{almost bipartite} if there exists a vertex $v\in V(G)$ such that $v\in V(C)$ for each odd cycle $C$.
A signed graph $(G,\Sigma)$ is said to be \emph{planar with two odd faces} if $G$ can be embedded in the plane such that all but two faces have a bounding cycle that is even.
Let $G$ be the graph obtained from two disjoint $K_3$s by connecting the vertices of one $K_3$ to the other $K_3$ by two parallel edges in a one-to-one way. From each parallel class with two edges, choose one edge, and let $\Sigma$ be the set of all these edges. We call the signed graph $(G,\Sigma)$ the \emph{double prism}. See Figure~\ref{fig:doubleprism}.

\begin{figure}
\begin{center}
\includegraphics[width=0.4\textwidth]{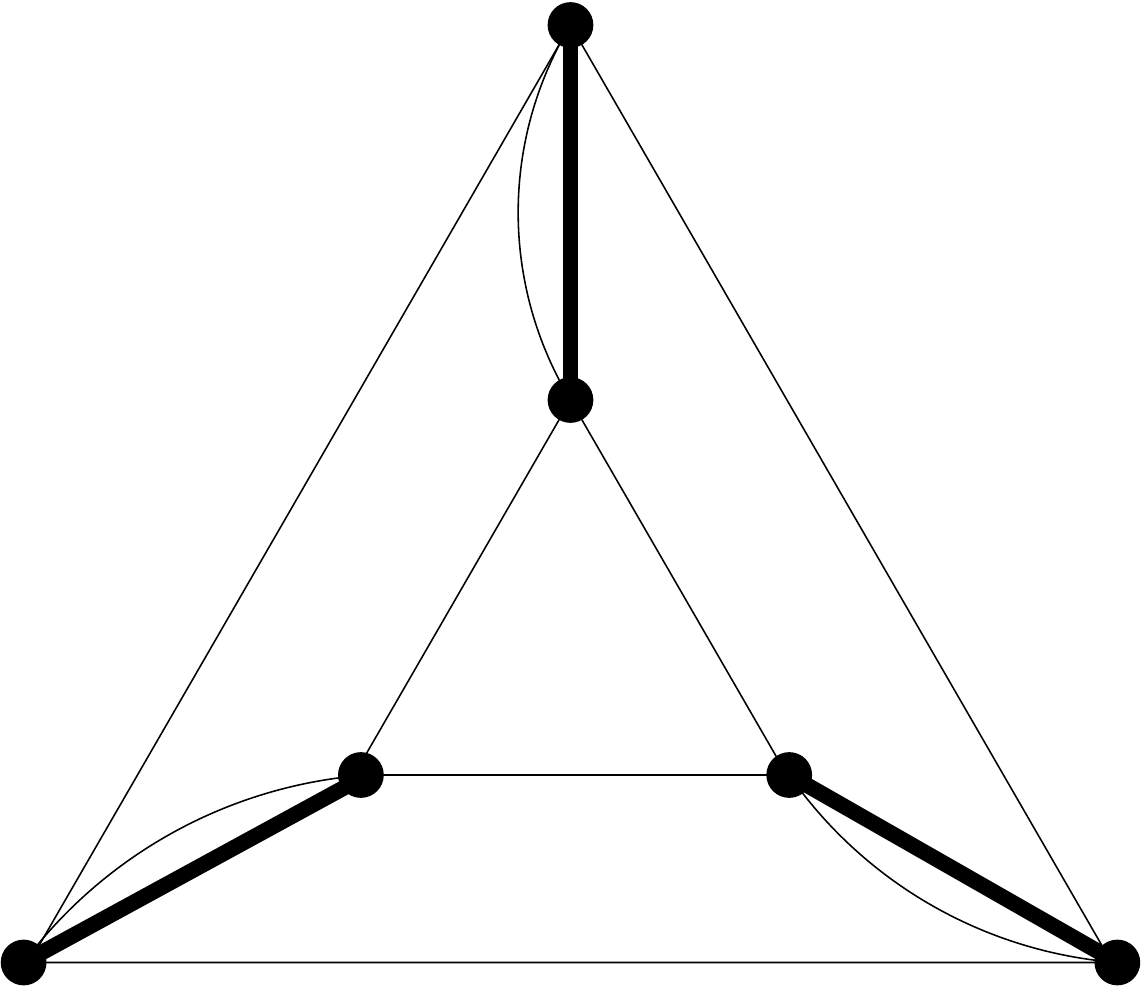}
\end{center}
\caption{The double prism}\label{fig:doubleprism}
\end{figure}

In \cite{Ger90a}, Gerards proved the following theorem; the theorem is stated in \cite{Ger90a} for signed graphs in which loops are permitted, but it is clear that the theorem still holds if we do not permit loops.

\begin{thm}\label{thm:gerards}
Let $(G,\Sigma)$ be a signed graph, with no $K_4^o$- and no $K_3^=$-minor. Then at least one of the following holds:
\begin{enumerate}
\item $(G,\Sigma)$ has a $0$-, $1$-, $2$-, or $3$-split;
\item $(G,\Sigma)$ is almost bipartite;
\item $G$ is planar with at most two odd faces;
\item\label{item:specgraph} $(G,\Sigma)$ is sign-equivalent to the double prism.
\end{enumerate}
\end{thm}

The idea of the proof of Theorem~\ref{thm:mainthm} is the following. Suppose for a contradiction that a signed graph $(G,\Sigma)$ with  no $K_4^o$- and no $K_3^=$-minor has $\nu(G,\Sigma)\geq 3$. From the results of Section~\ref{sec:nusplits}, we obtain that  if $(G,\Sigma)$ has a $0$-, $1$-, $2$-, or $3$-split, then some part $(H,\Omega)$ of this split has $\nu(H,\Omega)\geq 3$. Hence we may assume that $(G,\Sigma)$ has no $0$-, $1$-, $2$-, or $3$-split. Then $(G,\Sigma)$ is $2$-connected. By Theorem~\ref{thm:gerards}, $(G,\Sigma)$ is almost bipartite, planar with at most two odd faces, or sign-equivalent to the double prism. In Section~\ref{sec:boundingnullity} we will show that each signed graph $(G,\Sigma)$ belonging to at least one of these latter classes has $\nu(G,\Sigma)\leq 2$, which yields a contradiction.

It is interesting to note that the proof of the above theorem given by Gerards uses the decomposition theorem of regular matroids of Seymour \cite{Seymour80a}. The connection is the following. If $(G,\Sigma)$ is a signed graph, then the \emph{extended even cycle matroid}, $\mathcal{M}(G,\Sigma)$, of $(G,\Sigma)$ is the binary matroid respresented over $\mbox{GF}(2)$ by the matrix
\begin{equation*}
\begin{bmatrix}
1 & \chi_{\Sigma}\\
0 & M_G
\end{bmatrix}.
\end{equation*}
Here $\chi_{\Sigma}$ denotes the characteristic vector of $\Sigma$ as a subset of $E(G)$ and $M_G$ denotes the vertex-edge incident matrix over $\mbox{GF}(2)$ of $G$. 
If $\mathcal{N}$ is a minor of the matroid $\mathcal{M}(G,\Sigma)$, then $\mathcal{N}$ can be represented over $\mbox{GF}(2)$ by a matrix
\begin{equation*}
\begin{bmatrix}
1 & \chi_{\Omega}\\
0 & M_H
\end{bmatrix},
\end{equation*}
where $M_H$ denotes the vertex-edge incident matrix over $\mbox{GF}(2)$ of a graph $H$ and $\chi_{\Omega}$ denotes the characteristic vector of a subset $\Omega$ of $E(H)$, such that $(H,\Omega)$ is a minor of $(G,\Sigma)$.
Now, $\mathcal{M}(K_3^=)$ and $\mathcal{M}(K_4^o)$ are isomorphic to the Fano plane and the dual of the Fano plane, respectively. Since regular matroids are exactly those binary matroids with no minor isomorphic to the Fano plane or the dual of the Fano plane, $\mathcal{M}(G,\Sigma)$ is regular for any signed graph $(G,\Sigma)$ with no $K_4^o$- and no $K_3^=$-minor. Seymour showed in \cite{Seymour80a} that any regular matroid decomposes into graphic and cographic matroids, and copies of a certain $10$-element matroid. The translation of regular matroids to signed graphs is as follows: the graphic matroids yield signed graphs that are almost bipartite, cographic matroids yield planar signed graph with two odd faces, and the $10$-element matroid yields the double prism.

The result of Seymour also provides a polynomial-time algorithm for recognizing whether a binary matroid is regular or not, see \cite{Schrijver86}. Hence there exists a polynomial-time algorithm for recognizing whether or not a signed graph has no $K_4^o$- and no $K_3^=$-minor. Together with Theorem~\ref{thm:mainthm} this yields a polynomial-time algorithm for testing whether or not a signed graph $(G,\Sigma)$ has $\nu(G,\Sigma)\leq 2$.

\section{Planar graphs with two odd faces}

In this section we show that planar graphs with two odd faces can be reduced to $K_2^=$ using certain transformations and reductions. In Section~\ref{sec:nusplits}, we will see that these transformations and reductions do not decrease $\nu$, and therefore a planar graph with two odd faces $(G,\Sigma)$ has $\nu(G,\Sigma)\leq \nu(K_2^=) = 2$. In this section, loops are permitted in the graphs.
 
A \emph{degree-one reduction} in a graph means that we delete a vertex of degree one. A \emph{loop reduction} in a graph means that we delete a loop. A \emph{series reduction} in a graph means that we delete a vertex $v$ of degree two which has two neighbors and connect the two neighbors of $v$ by an edge. A \emph{parallel reduction} in a graph means that we delete all but one edge in a class of parallel edges. A \emph{$Y\Delta$-transformation} in a graph means that we delete a vertex of degree three which has three neighbors and add between each pair of these neighbors an edge. A \emph{triangle} in a graph is a subgraph isomorphic to $K_3$. A \emph{$\Delta Y$-transformation} means that we delete the edges of a triangle and add a new vertex and edges between this new vertex and the vertices of the triangle. A \emph{$\Delta Y$-exchange} is a $\Delta Y$- or $Y\Delta$-transformation.

We say that a graph $G$ with two particular distinct vertices, called \emph{terminals}, is \emph{$\Delta Y$-reducible} to  $G'$ if $G'$ can be obtained from $G$ by some sequence of degree-one, loop, series, and parallel reductions and $\Delta Y$-exchanges, where the terminals cannot be deleted by a reduction or transformation.

\begin{thm}[Epifanov \cite{Epifanov}]\label{thm:epifanov}
Any connected planar graph with two distinct terminals is $\Delta Y$-reducible to a single edge where the two terminals are the ends of this edge.
\end{thm}

A graph $G$ is called a \emph{block} if for every two distinct edges $e,f$ of $G$, there exists a cycle in $G$ containing $e$ and $f$. So, for example, a block with at least two edges cannot have any loops. A block with at least three vertices is $2$-connected. If $G$ is a plane block, then the dual of $G$ is also a block. (A plane block is a block that is embedded in the plane.) Furthermore, in a plane block a series reduction gives in the dual a parallel reduction; if the block has at least three edges, then a parallel reduction on two parallel edges that bound a face gives in the dual a series reduction; if the block consists of two parallel edges, then a parallel reduction on the two parallel edges gives in the dual a loop; and a $Y\Delta$-transformation gives in the dual a $\Delta Y$-transformation. However, a $\Delta Y$-transformation in a plane block does not necessarily give in the dual a $Y\Delta$-transformation. We call a $\Delta Y$-transformation on a triangle where each of the vertices of the triangle has degree at least three an \emph{allowable $\Delta Y$-transformation}. An allowable $\Delta Y$-transformation on a block gives a block. In a plane block, an allowable $\Delta Y$-transformation gives in the dual a $Y\Delta$-transformation. If a triangle has exactly one vertex, $v$, of degree two in $G$, we call the deletion of the edge in the triangle that is not incident to $v$ a \emph{$\Delta P_2$-reduction}. A $\Delta P_2$-reduction on a block gives a block. (Since the triangle on which we apply the $\Delta P_2$-reduction has exactly one vertex of degree two, the graph has at least four edges.) The dual reduction of a $\Delta P_2$-reduction is the following. If a vertex $v$ in a graph has degree two and three edges incident to it, then a \emph{series-parallel reduction} on $v$ means that we contract the edge incident to $v$ that has no edge parallel to it. In a plane block, a $\Delta P_2$-reduction gives in the dual a series-parallel reduction.
If there are at least two vertices of degree two in the triangle of a graph $G$ which is a block, then $G$ is equal to the triangle. A series reduction on the non-terminal vertex of triangle and then a parallel reduction yields a single edge where the two terminal are the ends of this edge.

We say that a block $G$ with two terminals is \emph{nicely $\Delta Y$-reducible} to  $G'$ if $G'$ can be obtained from $G$ by some sequence of series, parallel, $\Delta P_2$ reductions and allowable $\Delta Y$-transformations and $Y\Delta$-transformations, where the terminals cannot be deleted by a reduction or transformation, and where in each step of the sequence each graph is a block. We summarize as follows.

\begin{thm}\label{thm:planarpre}
Any planar block with two distinct terminals is nicely $\Delta Y$-reducible to a single edge where the two terminals are the ends of this edge.
\end{thm}

The last step in this sequence of series, parallel, $\Delta P_2$ reductions and allowable $\Delta Y$-transformations and $Y\Delta$-transformation is a parallel reduction. Hence we obtain the following theorem.

\begin{thm}\label{thm:planar}
Any planar block with two distinct terminals is nicely $\Delta Y$-reducible to two parallel edges where the two terminals are the ends of these edges.
\end{thm}

Let $G$ be a plane block with at least two edges. Any face of $G$ is enclosed by a cycle. For any face $F$, if $v$ is a vertex incident with $F$, then a series reduction on $v$ yields a face $\phi(F)$ whose enclosing cycle is obtained from the enclosing cycle of $F$ by a series reduction on $v$, and if $v$ is a vertex that is not incident with $F$, then a series reduction on $v$ yields the face $F$. Similar statements hold for series-parallel reductions and $Y\Delta$-transformations. 
If $F$ is bounded by two parallel edges, then a parallel reduction on these parallel edges will delete the face $F$. A parallel reduction on parallel edges that do not bound $F$ yields $F$.
Similarly, if $F$ is bounded by a triangle, then a $\Delta Y$-transformation on this triangle will delete $F$, and a $\Delta Y$-transformation on a different triangle yields $F$. If $F$ is not bounded by a triangle, then a $\Delta Y$-transformation yields $F$.

We say that a plane block $G$ with two distinct faces is \emph{nicely $\Delta Y$-reducible} to  a plane graph $G'$ if $G'$ can be obtained from $G$ by some sequence of series, parallel, and series-parallel reductions, and $\Delta Y$-exchanges, where the faces cannot be deleted by a reduction or transformation. 

Dualizing Theorem~\ref{thm:planar} yields the following theorem.

\begin{thm}\label{thm:dualplanar}
Any plane block with two distinct faces is $\Delta Y$-reducible to two parallel edges where the two faces correspond to the two distinct faces.
\end{thm}

Let $(H,\Omega)$ be a signed graph in which $H$ is a plane graph. We call a face of $(H,\Omega)$ odd if the enclosing cycle is odd, and call it even if the enclosing cycle is even. 


Let $(G,\Sigma)$ be a signed graph. If $v$ is a vertex of $(G,\Sigma)$ of degree two which has two neighbors, then a \emph{series reduction on $v$} in $(G,\Sigma)$ means that we delete $v$, connect the two neighbors of $v$ by an even edge if both edges incident to $v$ are odd, connect the two neighbors of $v$ by an odd edge if exactly one of the edges incident to $v$ is odd, and connect the two neighbors of $v$ by an even edge if both edges incident to $v$ are even. 
If two parallel edges in $(G,\Sigma)$ bound an even face $F$, then a \emph{parallel reduction on $F$} means that we delete one edge from these two parallel edges. Suppose $e$ and $f$ are parallel edges of the signed graph $(G,\Sigma)$ and $e\in \Sigma$ and $f\not\in \Sigma$. If an end, $v$, of $e$ is incident to exactly one edge $g\not\in \{e,f\}$, then by a \emph{parallel-series reduction} on $v$ in a signed graph we mean that we contract $g$. Notice that if $(G,\Sigma)$ is a plane graph with two odd faces, then a parallel-series reduction yields a plane graph with two odd faces. 
If $v$ is a vertex of $(G,\Sigma)$ of degree three which has three neighbors, then a \emph{$Y\Delta$-transformation (on $v$)} in $(G,\Sigma)$ means that we delete $v$ and for each pair of distinct neighbors, add an odd (even) edge if the path of length two through $v$ between this pair of neighbors is odd (even). 
If $T$ is an even triangle in $(G,\Sigma)$, then 
$T$ has either no or exactly two odd edges. If $T$ has no odd edges, then a $\Delta Y$-transformation on $T$ means that we remove the edges of $T$, add a new vertex $v$, and for each vertex $w$ in $T$ add an even edge between $v$ and $w$. If $T$ has exactly two odd edges, let $w_1$ be the vertex of $T$ that is incident to two odd edges, and let $w_2$ and $w_3$ be the other vertices of $T$. In this case, a $\Delta Y$-transformation on $T$  means that we remove the edges of $T$, add a new vertex $v$, and add an odd edge between $w_1$ and $v$, and even edges between $w_2$ and $v$ and between $w_3$ and $v$. We do not allow a $\Delta Y$-transformation on an odd triangle.

We say that a plane block $(G,\Sigma)$ with two odd faces is \emph{nicely $\Delta Y$-reducible} to  a plane block $(G',\Sigma')$ if $(G',\Sigma')$ can be obtained from $(G,\Sigma)$ by some sequence of series, parallel, and series-parallel reductions, and $\Delta Y$-exchanges. 

If we make the two odd faces the distinct two faces, then Theorem~\ref{thm:dualplanar} becomes our main result of this section.

\begin{thm}\label{thm:planarreducible}
Let $(G,\Sigma)$ be a plane block with two odd faces. Then $(G,\Sigma)$ is nicely $\Delta Y$-reducible to  $K_2^=$. 
\end{thm}

\section{Splits, reductions, and $\Delta Y$-exchanges in signed graphs}\label{sec:nusplits}

In this section, we show that $\nu(G',\Sigma') = \nu(G,\Sigma)$ if $(G',\Sigma')$ is obtained from $(G,\Sigma)$ by a series, parallel, or a series-parallel reduction, or a $\Delta Y$-exchange. Then we obtain formulas which allow to calculate $\nu(G,\Sigma)$ from the values of $\nu$ on the parts of a $0$-, $1$-, $2$-, or $3$-split of $(G,\Sigma)$. 

A proof of the following lemma can be found in \cite{AraHalLivdH2012}.
\begin{lemma}\label{lem:smnull1}
Let $(G,\Sigma)$ be a connected bipartite signed graph and let $A\in S(G,\Sigma)$ be positive semidefinite. If $x\in \ker(A)$ is nonzero, then $x$ has only nonzero entries. Furthermore, $\nullity(A)\leq 1$.
\end{lemma}

For a proof of the following lemma, see \cite{MR2517585}.
\begin{lemma}\label{lem:Mmatrix}
If $A$ is a positive definite matrix whose off-diagonal entries are nonpositive, then each entries of $A^{-1}$ is nonnegative.
\end{lemma}
%

Let $(G,\Sigma)$ be a signed graph, and let $G_1$ and $G_2$ be subgraphs of $G$. Let $\Sigma_i = \Sigma\cap E(G_i)$ for $i=1,2$. We call the pair $[(G_1,\Sigma_1),(G_2,\Sigma_2)]$ a \emph{$k$-partition} if $G_1\cup G_2 = G$, $E(G_1)\cap E(G_2)=\emptyset$, and $k=\lvert V(G_1)\cap V(G_2)\rvert$.

\begin{lemma}\label{lem:reduction2clique}
Let $(G,\Sigma)$ be a signed graph and let $[(G_1,\Sigma_1),(G_2,\Sigma_2)]$ be a $k$-partition of $(G,\Sigma)$. Suppose $(G_2,\Sigma_2)$ is bipartite and connected, and let $(H,\Omega)$ be obtained from $(G_1,\Sigma_1)$ by adding between each pair of distinct vertices $u,v$ of $V(G_1)\cap V(G_2)$ an even (odd) edge if there is an even (odd) path in $(G_2,\Sigma_2)$ between $u$ and $v$. Then $M_+(H,\Omega)\geq M_+(G,\Sigma)$ and $\nu(H,\Omega)\geq \nu(G,\Sigma)$.
\end{lemma}
\begin{proof}
Let $S = V(G_1)\cap V(G_2)$, $D=V(G_1)\setminus S$, and $C=V(G_2)\setminus S$.  As $(G_2,\Sigma_2)$ is bipartite, we can re-sign on $U\subseteq C$ such that each edge in $G[C]$ is even. 

Let $A\in S(G,\Sigma)$ be positive semidefinite.
Suppose for a contradiction that $\nullity A[C]\geq 1$; let $x\in \ker(A[C])$ be nonzero. Let $y$ be the vector on $V(G_2)$ defined by $y[C] = x$ and $y_s=0$ if $s\in V(G_2)\setminus V(G_1)$.
Since $y^T A[V(G_2)] y = 0$, $y \in \ker(A[V(G_2)])$. Since $(G_2, \Sigma_2)$ is bipartite and connected, this contradicts  Lemma~\ref{lem:smnull1}. Hence $A[C]$ is positive definite.

The Schur complement of $A[C]$ in $A$ is
\begin{equation*}
B = \begin{bmatrix}
A[D] & A[D,S]\\
A[S,D] & A[S]-A[S,C] A[C]^{-1} A[C,S]
\end{bmatrix}.
\end{equation*}
From Lemma~\ref{lem:Mmatrix} it follows that $B\in S(H,\Omega)$.
Since $A[C]$ is positive definite,  $B$ is positive semidefinite and $\nullity B = \nullity A$. Hence $M_+(H,\Omega)\geq M_+(G,\Sigma)$.

Suppose $A$ has the SAP. Let $Y=[y_{i,j}]$ be a symmetric matrix with $y_{i,j} = 0$ if $i$ and $j$ are adjacent in $H$ or if $i=j$, and such that $B Y = 0$. Let 
$Z = -A[C]^{-1}A[C,S] Y[S,D]$ and
\begin{equation*}
X = [x_{i,j}] = \begin{bmatrix}
Y[D] & Y[D,S] & Z^T\\
Y[S,D] & 0 & 0\\
Z & 0 & 0 
\end{bmatrix}.
\end{equation*}
Then $x_{i,j} = 0$ if $i$ and $j$ are adjacent in $G$ or if $i=j$, and $A X = 0$. Since $A$ has the SAP, $X=0$. Hence $Y=0$, which means that $B$ has the SAP. Hence $\nu(H,\Omega)\geq \nu(G, \Sigma)$.
\end{proof}

From Theorem~\ref{lem:reduction2clique}, we immediately obtain the following lemma.
\begin{lemma}\label{lem:YDelta}
Let $(G,\Sigma)$ be a signed graph. If $(G',\Sigma')$ is obtained from $(G,\Sigma)$ by applying a $Y\Delta$-transformation, then $M_+(G,\Sigma) \leq M_+(G',\Sigma')$ and $\nu(G,\Sigma)\leq \nu(G',\Sigma')$.
\end{lemma}

The following lemma deals with $\Delta Y$-transformations. Recall that we allow a $\Delta Y$-transformation only on an even triangle.

\begin{lemma}\label{lem:DeltaY}
Let $(G,\Sigma)$ be a signed graph. If $(G',\Sigma')$ is a signed graph obtained from $(G,\Sigma)$ by applying a $\Delta Y$-transformation, then $M_+(G,\Sigma) \leq  M_+(G',\Sigma')$ and $\nu(G,\Sigma)\leq \nu(G',\Sigma')$.
\end{lemma}
\begin{proof}
Let $A\in S(G,\Sigma)$ be a positive semidefinite matrix. Let $S=\{v_1,v_2,v_3\}$ be the vertices of the triangle $\Delta$ on which we apply the $\Delta Y$-transformation and let $C=V(G)\setminus S$. We can write $A$ as
\begin{equation*}
\begin{bmatrix}
K+T & A[S,C]\\
A[C,S] & A[C]
\end{bmatrix},
\end{equation*}
where $t_{i,j} < 0$ ($t_{i,j} > 0$) if the edge of the triangle $\Delta$ connecting $i$ and $j$ is even (odd), and where $k_{i,j} < 0$ if the other edges of $(G,\Sigma)$ connecting $i$ and $j$ are all even, $k_{i,j}>0$ if the other edges of $(G,\Sigma)$ connecting $i$ and $j$ are all odd, $k_{i,j}\in \mathbb{R}$ if the other edges of $(G,\Sigma)$ connecting $i$ and $j$ include both even and odd edges, and $k_{i,j}=0$ if there are no other edges of $(G,\Sigma)$ connecting $i$ and $j$.

We can find real numbers $b,c,d$ such that $t_{v_1,v_2} = -bc$, $t_{v_1,v_3} = -bd$, and $t_{v_2,v_3} = -cd$. To see this, notice that, as $\Delta$ is an even triangle, $t_{v_1,v_2} t_{v_1,v_3} t_{v_2,v_3} < 0$, and so $t_{v_1,v_2} t_{v_2,v_3}/t_{v_1,v_3}< 0$. Hence 
\begin{equation*}
c = \sqrt{-t_{v_1,v_2} t_{v_2,v_3}/t_{v_1,v_3}}
\end{equation*} 
is a well-defined real number.
The real numbers $b$ and $d$ can be found similarly.
Let
\begin{equation*}
B = [b_{i,j}] \begin{bmatrix}
1 & b & c & d & 0\\
b & t_{v_1,v_1}+b^2 & k_{v_1,v_2} & k_{v_1,v_3} & A[v_1,C]\\
c & k_{v_2,v_1} & t_{v_2,v_2}+c^2 & k_{v_2,v_3} & A[v_2,C]\\
d & k_{v_3,v_1} & k_{v_3,v_2} & t_{v_3,v_3}+d^2 & A[v_3,C]\\
0 & A[C,v_1] & A[C,v_2] & A[C,v_3] & A[C]
\end{bmatrix}.
\end{equation*}
Since $B/B[1] = A$ and $b_{1,1}>0$, $B$ is positive semidefinite and $\nullity{B}=\nullity{A}$. Furthermore, since $B\in S(G',\Sigma')$, $M_+(G',\Sigma')\geq M_+(G,\Sigma)$.

Suppose $A$ has the SAP. Let $Y = [y_{i,j}]$ be a symmetric matrix with $y_{i,j} = 0$ if $i$ and $j$ are adjacent or if $i=j$, and such that $BY=0$. We can write $Y$ as
\begin{equation*}
\begin{bmatrix}
0 & 0 & 0 & 0 & Y[v_0,C]\\
0 & 0 & y_{v_1,v_2} & y_{v_1,v_3} & Y[v_1,C]\\
0 & y_{v_1,v_2} & 0 & v_{v_2,v_3} & Y[v_2,C]\\
0 & y_{v_1,v_3} & y_{v_2,v_3} & 0 & Y[v_3,C]\\
Y[S,v_0] & Y[S,v_1] & Y[S,v_2] & Y[S,v_3] & Y[C]
\end{bmatrix}.
\end{equation*}
Then, since $B Y = 0$, $y_{v_1,v_2} = y_{v_1,v_3} = y_{v_2,v_3} = 0$. 
Let 
\begin{equation*}
X = [x_{i,j}] = \begin{bmatrix}
0 & 0 & 0 & Y[v_1,C]\\
0 & 0 & 0 & Y[v_2,C]\\
0 & 0 & 0 & Y[v_3,C]\\
Y[C,v_1] & Y[C,v_2] & Y[C,v_3] & Y[C]
\end{bmatrix}.
\end{equation*}
Then $A X = 0$ and $x_{i,j} = 0$ if $i$ and $j$ are adjacent or if $i=j$. Since $A$ has the SAP, $X=0$. Hence $Y=0$, which means that $B$ has the SAP. Hence $\nu(G',\Sigma')\geq \nu(G,\Sigma)$.
\end{proof}

The following lemma deals with series and series-parallel reductions.

\begin{lemma}\label{lem:seriesparallelred}
Let $(G,\Sigma)$ be a signed graph. If $(G_1, \Sigma_1)$ is obtained from $(G,\Sigma)$ by a series or series-parallel reduction, then $\nu(G_1,\Sigma_1) = \nu(G,\Sigma)$.
\end{lemma}
\begin{proof}
Since $(G_1,\Sigma_1)$ is sign-equivalent to a minor of $(G,\Sigma)$, $\nu(G_1,\Sigma_1)\leq \nu(G,\Sigma)$. Hence it suffices to show that $\nu(G,\Sigma)\leq \nu(G_1,\Sigma_1)$.

Let $A = [a_{i,j}]\in S(G,\Sigma)$ be a positive semidefinite matrix that has the SAP. Let $v$ be the vertex on which we apply the series or series-parallel reduction, let $u$ and $w$ be the neighbors of $v$, where we assume that between $v$ and $w$ there is a single edge, and let $R = V(G)\setminus\{u,v,w\}$. Then $a_{v,v} > 0$, for otherwise the $2\times 2$ principal matrix $\begin{bmatrix} a_{v,v} & a_{v,w}\\ a_{w,v} & a_{w,w}\end{bmatrix}$ would be negative, contradiction that $A$ is positive semidefinite.
The Schur complement of $A[v]$ in $A$ is
\begin{equation*}
B = [b_{i,j}] = \begin{bmatrix}
A[\{u,w\}-A[\{u,w\},v]a_{v,v}^{-1} A[v,\{u,w\}] & A[\{u,w\},R]\\
A[R,\{u,w\}] & A[R]
\end{bmatrix}.
\end{equation*}
Since $a_{v,v} > 0$, $B$ is positive semidefinite and $\nullity(B ) =\nullity(A)$. Since $b_{u,w} = a_{u,w} - a_{u,v}a_{v,v}^{-1} a_{v,w}$, $B\in S(G_1,\Sigma_1)$.

To see that $B$ has the SAP, let $Y=y_{i,j}$ be a symmetric matrix with $y_{i,j}=0$ if $i=j$ or if $i$ and $j$ are adjacent, and such that $B Y = 0$. Let $Z = -a_{v,v}^{-1} A[v,\{u,w\}] Y[\{u,w\},R]$ and
\begin{equation*}
X = [x_{i,j}] = \begin{bmatrix}
0 & 0 & Z\\
0 & 0 & Y[\{u,w\},R]\\
Z^T & Y[R,\{u,w\}] & Y[R]
\end{bmatrix}.
\end{equation*}
Then $x_{i,j} = 0$ if $i=j$ or if $i$ and $j$ are adjacent, and $A X = 0$. As $A$ has the SAP, $X=0$. Hence $Y=0$, which shows that $B$ has the SAP. Hence $\nu(G,\Sigma) =\nu(G_1,\Sigma_1)$.
\end{proof}

From the previous lemma it follows that the problem of determining $\nu(G,\Sigma)$ for a signed graph $(G,\Sigma)$ can be transformed, by subdividing each even edge, to determining $\nu(H,E(H))$ for a signed graph $(H,E(H))$, in which each edge is odd.

The proof of the following lemma is clear.
\begin{lemma}\label{lem:parallelred}
Let $(G,\Sigma)$ be a signed graph. If $e,f$ are parallel edges and either $e,f\in \Sigma$ or $e,f\not\in \Sigma$, and $(G',\Sigma')$ is the signed graph obtained from $(G,\Sigma)$ by deleting $e$, then 
$M_+(G',\Sigma')= M_+(G,\Sigma)$ and $\nu(G',\Sigma')= \nu(G,\Sigma)$.
\end{lemma}

We now show that if a signed graph $(G,\Sigma)$ has a $k$-split ($k=1,2,$ or $3$), then $\nu(G,\Sigma)$ can be calculated from the values of $\nu$ on the parts of the $k$-split. 

\begin{thm}\label{thm:3split}
Let $(G_1,\Sigma_1)$ form a $3$-split of $(G,\Sigma)$. Then $\nu(G,\Sigma) = \nu(G_1,\Sigma_1)$. 
\end{thm}
\begin{proof}
Since $(G_1,\Sigma_1)$ is a minor of $(G,\Sigma)$, $\nu(G_1,\Sigma_1)\leq \nu(G,\Sigma)$.

Let $E_1,E_2$ be a partitioning of $E(G)$ with both $E_1$ and $E_2$ nonempty which gives rise to the $3$-split $(G_1,\Sigma_1)$. 
For $i=1,2$, let $V_i$ be the set of all ends of edges in $E_i$, and let $\tilde{G_i} = (V_i,E_i)$. Let $(H,\Omega)$ be obtained from $(\tilde{G_1},E_1\cap \Sigma)$ by adding between each pair of distinct vertices $u,v$ of $V(G_1)\cap V(G_2)$ an even (odd) edge if there is an even (odd) path in $(\tilde{G_2},E_2\cap \Sigma)$ between $u$ and $v$. Let $T$ be the triangle on the three added edges. Since $(\tilde{G_2},E_2\cap \Sigma)$ is bipartite, $T$ is an even cycle, and moreover, since $(\tilde{G_2},E_2\cap \Sigma)$ is connected, $\nu(H,\Omega)\geq \nu(G,\Sigma)$, by Lemma~\ref{lem:reduction2clique}. 
Applying a $\Delta Y$-transformation on the triangle $T$ in $(H,\Omega)$ gives $(G_1,\Sigma_1)$. By Lemma~\ref{lem:DeltaY}, $\nu(G_1,\Sigma_1)\geq \nu(H,\Omega)$. Hence $\nu(G_1,\Sigma_1)\geq \nu(G,\Sigma)$, and since  $\nu(G_1,\Sigma_1)\leq \nu(G,\Sigma)$, the theorem follows.
\end{proof}

The following lemma shows that the SAP puts certain restrictions on a matrix $A\in S(G,\Sigma)$.

\begin{lemma} \label{lem:1eigenvalue0}
Let $(G,\Sigma)$ be a signed graph and let $A \in S(G,\Sigma)$ be positive semidefinite. Let $S \subseteq V$ and let $C_1,\dots,C_m$ be the vertex-sets of the connected components of
$G-S$. If $A$ has the SAP, then there is at most one $C_i$ with $\nullity(A[C_i]) > 0$.
\end{lemma}
\begin{proof}
Suppose that $\nullity(A[C_i])>0$ and $\nullity(A[C_j])>0$, where $i\not=j$. Let $y$ be a nonzero vector in $\ker(A[C_i])$ and $z$ be a nonzero vector in $\ker(A[C_j])$. Define the vectors $u,w$ on $V(G)$ by $u_v = y_v$ if $v\in C_i$ and $u_v=0$ otherwise, and $w_v = z_v$ if $v\in C_j$ and $w_v=0$ otherwise. Then $u^T A u = y^T A[C_i] y = 0$ and $w^T A w = z^T A[C_j] z = 0$. Hence $u,w \in \ker(A)$, as $A$ is positive semidefinite.
Let 
\begin{equation*}
X = [x_{i,j}] = u w^T + w u^T.
\end{equation*}
Then $x_{i,j}=0$ if $i$ and $j$ are adjacent or if $i=j$. Since $A X = 0$ and $X\not=0$, $A$ does not have the SAP. This contradiction shows that there is at most one $C_i$ with $\nullity(A[C_i]) > 0$.
\end{proof}

\begin{thm}\label{thm:12split}
Let $(G_1,\Sigma_1)$ and $(G_2,\Sigma_2)$ form a $0$-, $1$-, or $2$-split of $(G,\Sigma)$. Then $\nu(G,\Sigma) = \max\{\nu(G_1, \Sigma_1),\nu(G_2, \Sigma_2)\}$.
\end{thm}
\begin{proof}
As $(G_1,\Sigma_1)$ and $(G_2,\Sigma_2)$ are isomorphic to minors of $(G,\Sigma)$, it follows by the minor-monotonicity of $\nu$ that $\nu(G_1,\Sigma_1)\leq \nu(G,\Sigma)$ and $\nu(G_2,\Sigma_2)\leq \nu(G,\Sigma)$. Hence it suffices to show that $\nu(G,\Sigma) = \nu(G_1,\Sigma_1)$ or $\nu(G,\Sigma)=\nu(G_2,\Sigma_2)$.

Let $A \in S(G,\Sigma)$ be a positive semidefinite matrix that has the SAP. 

Suppose $(G_1,\Sigma_1)$ and $(G_2,\Sigma_2)$ form a $0$-split of $(G,\Sigma)$. By Lemma~\ref{lem:1eigenvalue0},
$A[V(G_1)]$ or $A[V(G_2)]$ is positive definite. By symmetry, we may assume that $A[V(G_1)]$ is positive definite. Then $A[V(G_2)]$ is positive semidefinite and $\nullity A[V(G_2)] = \nullity A$. To see that $A[V(G_2)]$ has the SAP, let $Y = [y_{i,j}]$ be a symmetric matrix with $y_{i,j} = 0$ if $i=j$ or if $i$ and $j$ are adjacent, and such that $A[V(G_2)] Y = 0$. Let
\begin{equation*}
X = [x_{i,j}] = \begin{bmatrix}
0 & 0\\
0 & Y
\end{bmatrix}.
\end{equation*}
Then, $x_{i,j} = 0$ if $i=j$ or if $i$ and $j$ are adjacent, and $AX = 0$. As $A$ has the SAP, $X = 0$. Hence $Y=0$, which shows that $A[V(G_2)]$ has the SAP.  Hence $\nu(G,\Sigma) =\nu(G_2,\Sigma_2)$.

Suppose now $(G_1,\Sigma_1)$ and $(G_2,\Sigma_2)$ form a $1$-or $2$-split of $(G,\Sigma)$. If $(G_1,\Sigma_1)$ and $(G_2,\Sigma_2)$ form a $2$-split, we may assume that the $2$-split is strong, for otherwise we can use Lemma~\ref{lem:reduction2clique}.
Let $S = V(G_1)\cap V(G_2)$, $C_1 = V(G_1)\setminus S$, and $C_2 = V(G_2)\setminus S$.
By Lemma~\ref{lem:1eigenvalue0}, $A[C_1]$ or
$A[C_2]$ is positive definite. We may assume that $A[C_1]$ is positive definite. The Schur complement of 
$A[V(C_1)]$ in
\begin{equation*}
A = \begin{bmatrix}
A[C_1] & A[C_1,S] & 0\\
A[S,C_1] & A[S] & A[S,C_2]\\
0 & A[C_2,S] & A[C_2]
\end{bmatrix}
\end{equation*} 
is 
\begin{equation*}
B = \begin{bmatrix}
A[S]-A[S,C_1] A[C_1]^{-1} A[C_1,S] & A[S,C_2]\\
A[C_2,S] & A[C_2]
\end{bmatrix}.
\end{equation*}
The matrix $B$ is positive semidefinite and $\nullity B = \nullity A$.
Furthermore, $B\in S(G_2,\Sigma_2)$.

To see that $B$ has the SAP, let $Y=y_{i,j}$ be a symmetric matrix with $y_{i,j}=0$ if $i=j$ or if $i$ and $j$ are adjacent, and such that $B Y = 0$. Let $Z = -A[C_1]^{-1} A[C_1,S] Y[S,C_2]$ and
\begin{equation*}
X = [x_{i,j}] = \begin{bmatrix}
0 & 0 & Z\\
0 & 0 & Y[S,C_2]\\
Z^T & Y[C_2,S] & Y[C_2]
\end{bmatrix}.
\end{equation*}
Then $x_{i,j} = 0$ if $i=j$ or if $i$ and $j$ are adjacent, and $A X = 0$. As $A$ has the SAP, $X=0$. Hence $Y=0$, which shows that $B$ has the SAP. Hence $\nu(G,\Sigma) =\nu(G_2,\Sigma_2)$.
\end{proof}

\section{Bounding the nullity of the graph classes}\label{sec:boundingnullity}

In this section we first show that $\nu(G,\Sigma)\leq 2$ for signed graphs $(G,\Sigma)$ that are almost bipartite or planar with at most two odd faces, and that $\nu(H,\Omega)=2$ for the double prism $(H,\Omega)$. Then we give the proof of Theorem~\ref{thm:mainthm}.
%

\begin{lemma}\label{lem:almostbipartite}
If $(G,\Sigma)$ is $2$-connected and almost bipartite, then $\nu(G,\Sigma)\leq M_+(G, \Sigma)\leq 2$.
\end{lemma}

\begin{proof}
Suppose for a contradiction that there exists a positive semidefinite matrix $A\in S(G,\Sigma)$ with $\nullity(A)>2$. Let $v$ be a vertex such that $v\in V(C)$ for each odd cycle $C$, and let $w$ be any other vertex. Then there exists a nonzero vector $x$ whose components corresponding to $v$ and $w$ equal $0$. Because $(G\setminus\{v\},\Sigma\setminus\delta(v))$ is connected and bipartite, we can deduce from Lemma~\ref{lem:smnull1} that the restriction of $x$ to $(G\setminus\{v\},\Sigma\setminus\delta(v))$ must have all positive or all negative components. However, 
$x_w = 0$, which yields a contradiction.
\end{proof}

\begin{lemma}\label{lem:twooddfaces}
If $(G,\Sigma)$ is $2$-connected planar with two odd faces, then $\nu(G,\Sigma)= 2$.
\end{lemma}
\begin{proof}
By Theorem~\ref{thm:planarreducible}, $(G,\Sigma)$ can  be reduced to $K_2^=$ using $\Delta Y$-exchanges, parallel, series, and parallel-series reductions. If $\nu(G,\Sigma)\geq 3$, then
	by Lemmas~\ref{lem:YDelta},~\ref{lem:DeltaY},~\ref{lem:seriesparallelred} and ~\ref{lem:parallelred}, $\nu(K_2^=)\geq 3$, which is a contradiction, as $\nu(K_2^=)=2$, see \cite{AraHalLivdH2012}. Hence $\nu(G,\Sigma)= 2$.
\end{proof}

\begin{lemma}\label{lem:doubleprism}
If $(G,\Sigma)$ is the double prism, then $\nu(G,\Sigma)=2$.
\end{lemma}
\begin{proof}
Applying $\Delta Y$-transformations on the two triangles in the double prism, then applying series-parallel reductions, and finally applying parallel reductions gives $K_2^=$. Hence $\nu(G,\Sigma) \leq \nu(K_2^=)$. Since $\nu(K_2^=)=2$ and since the double prism contains $K_2^=$ as a minor, $\nu(G,\Sigma)=2$.
\end{proof}

We now give the proof of Theorem~\ref{thm:mainthm}.

\begin{proof}[Proof of Theorem~\ref{thm:mainthm}]
Since $\nu(K_4^o) = \nu(K_3^=) = 3$, a signed graph $(G,\Sigma)$ with $\nu(G,\Sigma)\leq 2$ has no $K_4^o$- and no $K_3^=$-minor.

For the converse, suppose for a contradiction that $(G,\Sigma)$ is a signed graph with no $K_4^o$- and no $K_3^=$-minor, and $\nu(G,\Sigma)\geq 3$. Then $(G,\Sigma)$ has at least three vertices. We take $(G,\Sigma)$ with a minimal number of edges. By Theorems~\ref{thm:3split} and~\ref{thm:12split}, $(G,\Sigma)$ cannot have a $0$-, $1$-, $2$-, or $3$-split. Hence $(G,\Sigma)$ is $2$-connected. By Theorem~\ref{thm:gerards}, $(G,\Sigma)$ is almost bipartite or planar with two odd faces, or is the double prism.  By Lemmas~\ref{lem:almostbipartite},~\ref{lem:twooddfaces}, and~\ref{lem:doubleprism} we obtain a contradiction. Hence $\nu(G,\Sigma)\leq 2$.
\end{proof}

In \cite{AraHalLivdH2012}, we showed that $\nu(K_5^o)=4$ and $\nu(K_4^=)=4$. From Lemmas~\ref{lem:YDelta} and~\ref{lem:DeltaY} it follows that $\nu(H,\Omega)=4$ for any signed graph $(H,\Omega)$ that can be obtained from $K_4^=$ by a sequence of $\Delta Y$- and $Y\Delta$-transformations. We pose the following conjecture. 

\begin{conj}
A signed graph $(G,\Sigma)$ has $\nu(G,\Sigma)\leq 3$ if and only if $(G,\Sigma)$ has no minor isomorphic to $K_5^o$, $K_4^=$, or any signed graph $(H,\Omega)$ that can be obtained from $K_4^=$ by a sequence of $\Delta Y$- and $Y\Delta$-transformations.
\end{conj}

\bibliographystyle{plainnat}
\bibliography{../../biblio}

\end{document}